\documentclass[11pt]{article}
\usepackage{amsmath,amssymb,amsbsy,amsfonts,amsthm,latexsym,
            amsopn,amstext,amsxtra,euscript,amscd,amsthm}

\newtheorem{lem}{Lemma}
\newtheorem{lemma}[lem]{Lemma}

\newtheorem{thm}{Theorem}
\newtheorem{theorem}[thm]{Theorem}

\def\\{\cr}
\def\({\left(}
\def\){\right)}
\def\[{\left[}
\def\]{\right]}
\def\<{\langle}
\def\>{\rangle}

\def\cP{{\mathcal P}}

\def\eps{\varepsilon}

\begin{document}

\title{On shifted primes with large prime factors and their products}

\author{Florian Luca\\
{Mathematical Institute, UNAM}\\
{04510 Mexico DF, Mexico}\\ 
{\tt fluca@matmor.unam.mx}
\and
{Ricardo Menares}\\
{Instituto de Matem\'aticas}\\
{Pontificia Universidad Cat\'olica de Valpara\'{\i}so}\\
{Blanco Viel 596, Cerro Bar\'on, Valpara\'{\i}so, Chile}\\
{\tt ricardo.menares@ucv.cl}
\and
{Amalia~Pizarro-Madariaga} \\
{Departamento de Matem\'aticas} \\
{Universidad de Valpara\'{\i}so, Chile}\\
{\tt amalia.pizarro@uv.cl}
}

\date{\today}

\pagenumbering{arabic}

\maketitle

\begin{abstract}
We estimate from below the lower density of the set of prime numbers $p$ such that $p-1$ has a prime factor of size at least $p^c$, where $1/4 \le  c \leq 1/2$.
We  also establish upper and lower bounds on  the counting function  of the set of positive integers $n\le x$ with exactly $k$ prime factors, counted with or without multiplicity, such that the largest prime factor of ${\text{\rm gcd}}(p-1: p\mid n)$ exceeds $n^{1/2k}$. 
\end{abstract}
\tableofcontents 
\section{Introduction}

For  an integer $n$ put $P(n)$ for the maximum prime factor of $n$ with the convention that $P(0)=P(\pm 1)=1$. A lot of work has been done understanding 
the distribution of $P(p-1)$ for prime numbers $p$. The extreme cases $P(p-1)=2$ and $P(p-1)=(p-1)/2$ correspond to Fermat primes and Sophie-Germain primes, respectively. 
Not only do we not know if there are infinitely many primes of these kinds, but we do not know whether for each $\varepsilon>0$ arbitrarily small there exist infinitely many primes 
$p$ with $P(p-1)<p^{\varepsilon}$ or $P(p-1)>p^{1-\varepsilon}$. 

For a set ${\mathcal C}$ of positive integers and a positive real number $x$ we put ${\mathcal C}(x)={\mathcal C}\cap [1,x]$. Let 
$$
\cP_\eps := \{p \textrm{ prime : } P(p-1) \geq p^{1-\eps}\}, \quad \kappa(\eps) = \liminf_{x \to \infty} \frac{\#\cP_\eps(x)}{\pi(x)}.  
$$

Goldfeld  proved in \cite{Gold} that $\kappa(1/2) \geq 1/2 $.  It is not known whether  $\cP_{{1}/{2} }$ has a relative density, nor what this density could be in case it exists.  Fouvry  \cite{Fou85}, showed that there exists $\eps_0\in (0,1/3)$ such that $\kappa({\eps_0})>0$. Baker and Harman \cite{HaBa96}, found  $0<\eps_{1}<\eps_0$ such that $\cP_{\eps_{1}}$ is infinite.

In this article, we generalize Goldfeld's result in two  different directions. First, we estimate from below the lower  density of $\cP_{\eps}$ for all $\eps \in [1/2, 3/4]$. Secondly, we estimate the  counting function of the set of square free positive integers having prime divisors that, when shifted, share a large common prime factor.  Both questions are motivated by a technique used in \cite{BM} to bound from below the degree of the field of coefficients of newforms in terms of the level. A feature of the method in loc. cit. is that what is needed are values of $\eps$ such that $\kappa(\eps)$ is as large as possible.  Since $\kappa(\eps)$ is clearly an increasing function of $\eps$,  in contrast with the aforementioned works, which are focused in dealing with smaller and smaller  values of $\eps$, here we concentrate on the case where this parameter is bigger than $1/2$.

We obtain the following results.

\begin{theorem}
\label{primo}
Let $0 \leq \alpha  \le {1}/{4}$.  Let 
$$N_\alpha= \{  p \textrm{  prime such that } P(p-1) \geq p^{{1}/{2} -\alpha} \}.$$
Then 
$$\# N_\alpha(x) \geq \left(\frac{1}{2} +\alpha\right) \frac{x}{\log x} + E(x);\quad E(x)=\left\{\begin{matrix} {\displaystyle{O\left(\frac{x\log\log x}{(\log x)^2}\right)}} &  (\alpha<1/4)\\
{\displaystyle{O\left(\frac{x}{(\log x)^{5/3}}\right)}} &  (\alpha=1/4).\\
\end{matrix}\right.
$$ 
The implied constant depends on $\alpha$. In particular,
$$
\kappa(1/2+\alpha)\ge 1/2+\alpha\quad {\text{for~all}}\quad \alpha\in [0,1/4].
$$
\end{theorem}

The case $\alpha=0$ is Goldfeld's result mentioned  above. Our proof of Theorem \ref{primo}  follows closely his method.

For any  $k\ge 1$ and $a\in (0,1/k)$, let 
$$
{\mathcal A}_{k,a}=\{n=p_1\cdots p_k, P\big(\gcd(p_1-1,\ldots,p_k-1)\big)>n^{a}\}.
$$
By Goldfeld's result,  $\#{\mathcal A}_{1,1/2}(x)\asymp x/\log x$.  Here, we prove the following result.

\begin{theorem}\label{compuesto}
If $k\ge 2$ and $a\in \big[1/(2k),17/(32k)\big)$ are fixed, then 
\begin{equation}
\label{eq:main}
\frac{x^{1-a(k-1)}}{(\log x)^{k+1}}\ll \#{\mathcal A}_{k,a}(x)\ll \frac{x^{1-a(k-1)}(\log\log x)^{k-1}}{(\log x)^2}.
\end{equation}
\end{theorem} 
The case $a=1/(2k)$ is important for the results from \cite{BM}. We have the estimate 
\begin{equation}\label{crudo}
\#{\mathcal A}_{k,1/(2k)}(x)=x^{1/2+1/2k+o(1)}, \quad x\to\infty.
\end{equation} 

 Goldfeld's method does not seem to extend to the situation in Theorem \ref{compuesto} (see the last section). Instead, we follow a more direct method.  For the lower bound, we rely on a  refined version of the Brun-Titchmarsh inquality due to Banks and Shparlinsky  \cite{BS}.

We remark that both theorems presented here remain valid if, instead of considering  large factors of $p-1$, we look at large factors $p+n$ for an arbitrary nonzero fixed integer $n$. 

We leave as a problem for the reader to determine the exact order of magnitude of $\#{\mathcal A}_{k,a}(x)$, or an asymptotic for it.

Throughout this paper, we use $p,~q,~r$ with or without subscripts for primes. We use the Landau symbols $O,~o$ and the Vinogradov symbols $\ll$ and $\gg $ with their regular meaning.
The constants implied by them might depend on some other parameters such as $\alpha,~k,~\varepsilon$ which we will not indicate. 

\section{Proof of Theorem \ref{primo}}

We follow Goldfeld's general strategy. Put $c={1}/{2}-\alpha$. So, ${1}/{4} \le  c \leq {1}/{2}$.   Let  $$N_c'(x)=\#\{ p \leq x : p \textrm{ is prime and } P(p-1) \geq x^c \}.$$
Since $\#N_\alpha(x) \geq N_c'(x)$, it is enough to give a lower bound for $N_c'(x)$.
Put 
$$
M_c(x)= \sum_{p \leq x} \sum_{\substack{\ell\mid p-1\\ \ell \geq x^c}} \log \ell,$$ where $p$ and $\ell$ denote primes. Since 
$$
\sum_{\substack{\ell \mid p-1 \\ \ell \geq x^c}} \log \ell \left\{ \begin{array}{ll}
= 0, & \textrm{ if } P(p-1) < x^c; \\
\leq \log x, & \textrm{ otherwise }
\end{array} \right.$$
we have that  
$$
M_c(x) \leq \log x \sum_{\substack{p \leq x \\ P(p-1) \geq x^c}} 1 = N_c'(x) \log x.
$$
Hence, $N_c'(x) \geq M_c(x)/\log x.$ Then, in order to prove Theorem \ref{primo}, it is enough to show that 
\begin{equation}\label{primera reduccion}
M_c(x) = (1-c)x + F(x),\quad F(x)=\left\{\begin{matrix} {\displaystyle{O_c\left(\frac{x\log \log x}{\log x} \right)}}, & (c>1/4);\\
{\displaystyle{O\left(\frac{x}{(\log x)^{2/3}}\right)}}, & (c=1/4).\\ 
\end{matrix}\right.
\end{equation}

We denote by $\Lambda(\cdot)$ the von Mangoldt's function. As usual, $\pi(x;b,a)$ is the number of primes $q\le x$ in the arithmetic progression $q\equiv a\pmod b$.
We define
$$
L(x;u,v)=\sum_{u<m\le v} \Lambda(m)\pi(x;m,1).
$$

\begin{lemma}
\label{log}
Assume $1/4\le c\le 1/2$. Then
$$
L(x;x^c,x) = M_c(x) + O\left(\frac{x^{{7}/{6} -{2c}/{3}}}{(\log x)^r}\right),
$$
where $r=0$ when $c>1/4$ and $r=2/3$ when $c=1/4$.
\end{lemma}

\begin{proof} 
Let $0<d<1-c$ be a real number and $r \in (0,1)$.  We assume that $x$ is large enough so that the inequality $x^{1-d}(\log x)^{r}<x$ holds. We put
\begin{eqnarray*}
M_1^d(x) & = & \sum_{\substack{x^c < \ell^k \leq x^{1-d}(\log x)^{r} \\ \ell \textrm{ prime, } k\geq 2 }} \pi (x;\ell^k,1)\log \ell\\
M_2^d(x) & = & \sum_{\substack{x^{1-d}(\log x)^{r} < \ell^k \leq x \\ \ell \textrm{ prime, } k\geq 2 }}  \pi (x;\ell^k,1)\log \ell.
\end{eqnarray*}
Hence, 
\begin{equation}\label{sumas}
L(x;x^c,x)-M_c(x) = M_1^d(x) + M_2^d(x).
\end{equation}
Using  the Brun-Titchmarsh inequality, we have that
\begin{eqnarray*}
M_1^d(x)  & \ll &  \frac{x}{\log x } \sum_{\substack{x^c < \ell^k \leq x^{1-d}(\log x)^{r} \\ \ell \textrm{ prime, } k\geq 2 }}  \frac{ \log \ell}{\ell^{k-1}(\ell-1)} \\
& \le &  \frac{x}{\log x}  \sum_{\substack{\ell \leq x^{(1-d)/2}(\log x)^{r/2}\\ \ell \textrm{ prime}}} 2\log\ell \sum_{k\ge c\log x/\log \ell} \frac{1}{\ell^k}\\
 & \le & \frac{x}{\log x}   \sum_{\ell \leq x^{(1-d)/2}(\log x)^{r/2}} \frac{4\log x}{x^c}\\
 & = & 4x^{1-c}\pi\left(x^{(1-d)/2}(\log x)^{r/2}\right)\\
 & \ll & \frac{x^{1-c+(1-d)/2}}{(\log x)^{1-r/2}}.
\end{eqnarray*}
On the other hand, for an integer $m>x^{1-d}(\log x)^{\alpha}$, we have that 
$$
\pi (x;m,1) < \sum_{\substack{n \leq x\\  n \equiv 1 \pmod m }} 1 \leq \frac{x}{m} < \frac{x^d}{(\log x)^{\alpha}}.
$$ 
Hence, 
\begin{eqnarray*}
M_2^d(x) & < & \frac{x^d}{(\log x)^{\alpha}} \sum_{\substack{x^{1-d}(\log x)^{\alpha} < \ell^k \le x \\ \ell \textrm{ prime, } k\geq 2 }}  \log \ell \\
& \ll &  \frac{x^d}{(\log x)^{\alpha}} (\log x) \pi(\sqrt{x}) \ll \frac{x^{d + \frac{1}{2} }}{(\log x)^{\alpha}}.
\end{eqnarray*}
Using \eqref{sumas}, we obtain 
$$
L_c(x)-M_c(x)= O\left( \frac{x^{1-c+(1-d)/2}}{(\log x)^{1-r/2}} + \frac{x^{d + \frac{1}{2}}}{(\log x)^{r}}\right).$$
We take $d=2/3(1-c)$ and then both exponents of $x$ above are equal and evaluate to $7/6-2/3 c$. Taking $r=0$ when $c<1/4$ and $r=2/3$ when $c=1/4$, we obtain the desired estimate.  
\end{proof}

\begin{lemma}
\label{BT}
Assume that $c\in (0,1/2]$. Then, for $B>0$,  we have 
$$
L\left(x;{x^c}/{(\log x)^B}, x^c\right) =O\left(\frac{x \log \log x}{\log x} \right), \quad (x \rightarrow \infty).
$$
 \end{lemma}

 \begin{proof} 
 This follows immediately from the Brun-Titchmarsh inequality (see, for example, equation (3) in \cite{Gold}).
 \end{proof}

\begin{lemma}\label{BV}
Assume that $c\in (0,1/2]$. Then, there exists $B>0$ such that
$$
L\left(x;1,{x^c}/{(\log x)^B}\right) = cx+O\left(\frac{x \log \log x}{\log x} \right), \quad (x \rightarrow \infty).
$$
\end{lemma}

\begin{proof} This follows easily from the Bombieri-Vinogradov theorem (see, for example, equation (2) in \cite{Gold}).
\end{proof}

\medskip 

\noindent {\it Proof of Theorem \ref{primo}}: We have (see p. 23 in \cite{Gold}),

\begin{equation}\label{basico}
 L(x;1,x) = x + O\left(\frac{x}{\log x}\right), \quad (x \rightarrow \infty).
\end{equation}
Take $B>0$ as in Lemma \ref{BV}. 
Since 
$$
L(x;1,x) =L\left(1,\frac{x^c}{(\log x)^B}\right)  + L\left(\frac{x^c}{(\log x)^B}, x^c\right) + L(x;x^c,x),
$$
the result follows by combining \eqref{primera reduccion} and Lemmas \ref{log}, \ref{BT} and \ref{BV}.
\qed

\section{Proof of Theorem \ref{compuesto}}

\subsection{The upper bound}

Let $x$ be large. It is sufficient to prove  the upper bound indicated at \eqref{eq:main} for the number of integers $n\in {\mathcal A}_{k,a}\cap [x/2,x]$, since then the upper bound will follow by changing $x$ to $x/2$, then to $x/4$ and so on, and summing up the resulting estimates.
So, we assume that $n\ge x/2$ is in ${\mathcal A}_{k,a}(x)$. 
Then $n=p_1\cdots p_k\le x$, where $p_1\le p_2\le \cdots \le p_k$, and $p_i=p\lambda_i+1$ for $i=1,\ldots,k$, where 
$$
p>n^{a}>(x/2)^{a}.
$$  
Note that
$$
p^k \lambda_1\cdots \lambda_k\le \phi(n)<n<x.
$$
Thus, $p<x^{1/k}$. Let ${\mathcal B}_1(x)$ be the set of such $n\le x$ such that $\lambda_k\le x^{\delta}$, where $\delta=\delta_{k}=15(k-1)/(32 k^2)$. Since $\lambda_1\le \cdots \le \lambda_k$, we get that $\lambda_i\le x^{\delta}$ for all $i=1,\ldots,k$. This shows that
\begin{equation}
\label{eq:B2}
\#{\mathcal B}_1(x)\le \pi(x^{1/k}) (x^{\delta})^k<x^{1/k+15(k-1)/(32k)}=o(x^{1-a(k-1)})\quad (x\to\infty),
\end{equation}
where we used the fact that $1/k+15(k-1)/(32k)<1-a(k-1)$, which holds for all $k\ge 2$ and $a\in (0,17/(32k))$.

From now on, we assume that $n\in {\mathcal B}_2(x)=\left({\mathcal A}_k\cap [x/2,x]\right)\backslash {\mathcal B}_1(x)$. Fix  the primes $p_1\le \cdots\le p_{k-1}$. Then $p$ is fixed, $p_k\le x/(p_1\ldots p_{k-1})$ and $p_k\equiv 1\pmod p$. The number of such primes is, by the Brun-Titchmarsch theorem (see \cite{MV}), at most
$$
\pi(x/(p_1\ldots p_{k-1}); p,1)\le \frac{2x}{(p-1) p_1\ldots p_{k-1} \log(x/(pp_1\ldots p_{k-1}))}.
$$
Since $x/(p p_1\ldots p_{k-1})>\lambda_k>x^{\delta}$, we get that the last bound is at most
$$
\ll \frac{x}{(\log x)pp_1\ldots p_{k-1}}.
$$
Keeping $p$ fixed and summing up the above bound over all ordered $k-1$-tuples of primes $(x/2)^{a}<p_1\le \cdots\le p_{k-1}\le x$ such that $p_i\equiv 1\pmod p$ for $i=1,\ldots,k-1$, we get a bound of 
\begin{equation}
\label{eq:bound}
\frac{x}{(\log x)p}\left(\sum_{\substack{q\le x\\ q\equiv 1\pmod p}} \frac{1}{q}\right)^{k-1}\ll
\frac{x(\log\log x)^{k-1}}{(\log x)p^k},
\end{equation}
where we used the fact that
$$
\sum_{\substack{q\le x\\ q\equiv 1\pmod p}} \frac{1}{q}\ll \frac{\log\log x}{p}
$$
uniformly in $(x/2)^{a}\le p\le x^{1/k}$, which follows from the Brun-Titchmarsch theorem by partial summation. Summing up the above bound \eqref{eq:bound} over all $p>(x/2)^{a}$ gives
\begin{eqnarray}
\label{eq:B3}
\# {\mathcal B}_2(x) & \ll & \frac{x(\log\log x)^{k-1}}{\log x} \sum_{(x/2)^{a}<p\le x^{1/k}} \frac{1}{p^k}\nonumber\\
&  \ll &   \frac{x(\log\log x)^{k-1}}{\log x}  \int_{(x/2)^{a}}^{x^{1/k}} \frac{d \pi(t)}{t^k} \nonumber\\
& \ll & \frac{x(\log\log x)^{k-1}}{\log x}  \left( \frac{1}{t^{k-1} \log t} \Big|_{t=(x/2)^{a}}^{t=x^{1/k}}+\int_{(x/2)^{a}}^{x^{1/k}} \frac{dt}{t^k \log t}\right)\nonumber\\
&  \ll &  \frac{x(\log\log x)^{k-1}}{\log x}\left(\frac{1}{x^{a(k-1)}\log x}\right)\nonumber\\
& \ll & \frac{x^{1-a(k-1)}(\log\log x)^{k-1}}{(\log x)^2}.
\end{eqnarray}
The upper bound follows from \eqref{eq:B2} and \eqref{eq:B3}.

\subsection{The lower bound}

The following result is Lemma 2.1 in \cite{BS}. 

\begin{lemma}
\label{lem:BS}
There exist functions $C_2(\nu)>C_1(\nu)>0$ defined for all real numbers $\nu\in (0,17/32)$ such that for every integer $u\ne 0$ and positive real number $K$, the inequalities 
$$
\frac{C_1(\nu)y}{p\log y}<\pi(y;p,u)<\frac{C_2(\nu) y}{p\log y}
$$
hold for all primes $p\le y^{\nu}$ with $O(y^{\nu}/(\log y)^K)$ exceptions, where the implied constant depends on $u,~\nu,~K$. Moreover, for any fixed $\varepsilon>0$, these functions can be chosen to satisfy the following properties:
\begin{itemize}
\item $C_1(\nu)$ is monotonic decreasing, and $C_2(\nu)$ is monotonic increasing;
\item $C_1(1/2)=1-\varepsilon$ and $C_2(1/2)=1+\varepsilon$.
\end{itemize}
\end{lemma}

So, we take $y=x^{1/k}$ and consider primes $p\in {\mathcal I}=[y^{ak},2y^{ak}]$. Then $2y^{ak}=y^{\nu}$, where $\nu=ak+(\log 2)/(ak\log y)<17/32$ for all $x$ sufficiently large. So, let $\varepsilon>0$ be such that
$a<17/32-\varepsilon$ and assume that $x$ is sufficiently large such that $\log 2/(\log y)<\varepsilon/2$. 
Then, by Lemma \ref{lem:BS} with $u=1$ and $K=2$, the set ${\mathcal P}$ of primes $p\le 2y$ such that 
$$
\pi(y;p,1)>\frac{C_1(17/32-\varepsilon/2) y}{p\log y }
$$
contains all primes $p\le 2y^{ak}$ with $O(y^{ak}/(\log y)^2)$ exceptions. Thus, the number of primes $p\in {\mathcal P}\cap {\mathcal I}$ satisfies
$$
\#\left({\mathcal P}\cap {\mathcal I}\right)\ge \pi(2y^{ak})-\pi(y^{ak})-O\left(\frac{y^{a}}{(\log y)^2}\right)>\frac{y^{ak}}{\log y}
$$
for all $x$ sufficiently large independently in $k$ and $a$. Consider numbers of the form  $n=p_1\cdots p_k$, where $p_1<\cdots<p_k\le y$ are all primes congruent to $1$ modulo $p$. Furthermore, it is clear that
$p=P(p_i-1)$ for $i=1,\ldots,k$. Note that $n\le x$. The number of such $n$ is, 
for $p$ fixed,
$$
\binom{\pi(y;p,1)}{k}\gg \left(\frac{y}{p\log y}\right)^k\gg \frac{x}{p^k (\log x)^k}.
$$
Summing up the above bound over $p\in {\mathcal P}\cap {\mathcal I}$, we get that
\begin{eqnarray*}
\#{\mathcal A}_{k,a}(x) & \gg & \frac{x}{(\log x)^k} \sum_{p\in {\mathcal P}\cap {\mathcal I}} \frac{1}{p^k}\gg 
\frac{x}{(\log x)^k} \left(\frac{\#\left({\mathcal P}\cap {\mathcal I}\right)}{y^{ak^2}}\right)\\
& \gg &  \frac{xy^{ak}}{y^{ak^2} (\log x)^{k}\log y}\gg \frac{x^{1-a(k-1)}}{(\log x)^{k+1}},
\end{eqnarray*}
which is what we wanted.

\section{Comments and Remarks}

It is not likely that Goldfeld's method extends to the situation considered in Theorem \ref{compuesto}. As we have seen, the proof of Theorem \ref{primo} is based on the identity \eqref{basico}. Then, Mertens's theorem, the Brun-Titchmarsh inequality and the Bombieri-Vinogradov theorem are used to extract the desired estimate out of it. If we try to follow the same strategy to prove Theorem \ref{compuesto}, for example with $a=1/(2k)$, we are then led to replace the left hand side of \eqref{basico} by
$$
L_k(x):= \sum_{m \leq x^{1/k} }^{ } \Lambda(m) \pi_k(x;m,1), 
$$
where $\pi_k(x;m,1) = \# \{ n \in \mathcal{A}_k(x) : p | n \Rightarrow p \equiv 1 \mod m\}.$ Let $\pi_k(x)$ denote the number of squarefree integers up to $x$ having exactly $k$ prime factors. Then, letting $p_1, p_2, \ldots, p_k$ denote primes, 
\begin{eqnarray}
 L_k(x) &=& \sum_{\substack{p_1< p_2<\cdots< p_k \\ p_1p_2 \cdots p_k  \leq x}}^{ } \sum_{\substack{m\mid {\text{\rm{gcd}}} (p_i-1) \\ 1\le i\le k}}\Lambda(m)   \nonumber \\
 &=& \sum_{\substack{p_1 <p_2 <\cdots < p_k \\ p_1p_2 \cdots p_k  \leq x}}^{ }  \log \left( {\text{\rm{gcd }}}\left(p_i-1:1\le i\le k \right)\right)  \nonumber \\
 &\geq & (\log 2) \pi_k(x) \gg_k \frac{x(\log \log x)^{k+1}}{\log x}, \quad x \rightarrow \infty. \nonumber
\end{eqnarray}
In view of \eqref{crudo},  we see that $L_k(x)$ grows much faster, when $k \geq 2$, than the counting function we are interested in. Hence, it is unlikely that  $L_k(x)$ can be used to obtain information on the growth of $\mathcal{A}_k(x)$.

\medskip 

{\bf Acknowledgement.} We thank N. Billerey for stimulating questions. Part of this work was done during a visit of F.~L. 
at the Mathematics Department of the Universidad de Valparaiso during 2013 with a MEC project from CONICYT. He thanks the people of this Department for their 
hospitality. F.~L. was also partially supported by a Marcos Moshinsky Fellowship and
Projects PAPIIT IN104512, 
CONACyT  163787 and CONACyT 193539. R.M. is partially supported by FONDECYT grant 11110225.



\end{document}